\documentclass{article}

\usepackage{amsmath}
\usepackage{amssymb}
\usepackage[margin=1in]{geometry}
\usepackage{dsfont}
\usepackage{amsthm}
\usepackage{enumerate}
\usepackage{color}
\usepackage[normalem]{ulem}
\usepackage[backend=biber, style=numeric-comp, doi=false,isbn=false,url=false, giveninits=true, date=year, maxbibnames=5]{biblatex}

\newcommand{\M}[2]{M_{#1\times #1}(\mathds{#2})}
\newcommand{\Msymm}[2]{M_{#1\times #1}^{\mathrm{symm}}(\mathds{#2})}
\newcommand{\sip}[1]{\langle#1\rangle}
\newcommand{\norm}[1]{\left\lVert#1\right\rVert}
\newcommand{\A}{\mathcal{A}}

\newcommand{\N}{\mathds{N}}

\newcommand{\C}{\mathds{C}}
\newcommand{\R}{\mathds{R}}

\newcommand{\Q}{\mathcal{Q}}
\newcommand{\diag}{\mathrm{diag}}
\newcommand{\Tr}{\mathrm{Tr}}

\newcommand{\blindfootnote}[1]{{\renewcommand\thefootnote{}\footnotetext{#1}}}

\numberwithin{equation}{section}

\newtheorem{theorem}{Theorem}[section]
\newtheorem{prop}[theorem]{Proposition}
\newtheorem*{prop*}{Proposition}
\newtheorem{lemma}[theorem]{Lemma}
\newtheorem{corollary}[theorem]{Corollary}

\theoremstyle{definition}
\newtheorem{definition}[theorem]{Definition}
\newtheorem*{remark}{Remark}

\newtheorem{problem}{Problem}

\newtheorem{example}[theorem]{Example}

\title{On the existence of derivations as square roots of generators of state-symmetric quantum Markov semigroups\blindfootnote{\emph{MSC2020}: 46L57. \emph{Keywords}: quantum Markov semigroups, derivations, states.\\
MV is supported by the NWO Vidi grant VI.Vidi.192.018 `Non-commutative harmonic analysis and rigidity of operator
algebras'.}}
\author{Matthijs Vernooij}

\AtEndDocument{\bigskip{\footnotesize%
  \textsc{Delft University of Technology, Faculty EEMCS/DIAM, P.O.Box 5031, 2600 GA Delft, The Netherlands} \par  
  \textit{E-mail address}: \texttt{m.n.a.vernooij@tudelft.nl}
}}

\begin{document}
\maketitle

\begin{abstract}
	Cipriani and Sauvageot have shown that for any $L^2$-generator $L^{(2)}$ of a tracially symmetric quantum Markov semigroup on a C*-algebra $\A$ there exists a densely defined derivation $\delta$ from $\A$ to a Hilbert bimodule $H$ such that $L^{(2)}=\delta^*\circ \overline{\delta}$. Here we show that this construction of a derivation can in general not be generalised to quantum Markov semigroups that are symmetric with respect to a non-tracial state. In particular we show that all derivations to Hilbert bimodules can be assumed to have a concrete form, and then we use this form to show that in the finite-dimensional case the existence of such a derivation is equivalent to the existence of a positive matrix solution of a system of linear equations. We solve this system of linear equations for concrete examples using Mathematica to complete the proof.
\end{abstract}

\section{Introduction}
In 1976 Lindblad argued that a quantum Markov semigroup (QMS) is a good way to describe irreversible processes in a quantum system \cite{Lindblad1976}. Since then QMSs have been extensively researched. Here we will focus on one line of investigation, which is due to Cipriani and Sauvageot \cite{CIPRIANI200378}. Let $(\mathcal{P}_t)_{t\geq0}$ be a QMS on a unital C*-algebra $\A$ with trace $\tau$, $L$ the generator of the QMS, and assume that each $\mathcal{P}_t$ is symmetric with respect to $\tau$. In this case $L$ can also be extended to a closed, densely defined, nonnegative self-adjoint operator on $L^2(\A,\tau)$, which we denote by $L^{(2)}$. Then Cipriani and Sauvageot showed that there exists a densely defined derivation $\delta$ from $\A$ to some Hilbert bimodule $H$ such that we have $L^{(2)}=\delta^*\circ\overline{\delta}$ when we consider $\delta$ as an operator on $L^2(\A,\tau)$ and use this to define the adjoint $\delta^*$ of $\delta$ \cite{CIPRIANI200378}. This construction of a derivation is widely used in non-commutative potential theory \cite{CIPRIANI200378}, non-commutative optimal transport theory \cite{Wirth2018,Hornshaw2020}, deformation-rigidity theory for von Neumann algebras \cite{Caspers2021}, and investigations into the decoherence time of quantum systems \cite{bardet2017}.

Most of these papers are only able to treat the tracially symmetric case \cite{CIPRIANI200378,Wirth2018,Hornshaw2020,bardet2017}. However, many von Neumann algebras have natural non-tracial states. In fact, type III von Neumann algebras do not even have any semifinite faithful normal trace \cite{Takesaki2002I}, which is a requirement for the construction of Cipriani and Sauvageot. However, these algebras have several physical applications, both in quantum field theory \cite{Yngvason2005} and in quantum statistical mechanics for infinite systems \cite{Bratteli1981}. Additional interest in type III factors can be found in free probability by looking at free Araki-Woods factors \cite{Shlyakhtenko1997} or in quantum group theory, for instance by considering free orthogonal quantum groups \cite{Daele1996}. However, even in other types of von Neumann algebras one might want to consider a non-tracial state, for example when studying generators that commute with the modular automorphism group with respect to that state, which is a trivial condition for tracial states \cite{CARLEN20171810}. 

This prompts us to pose the following problem, which has already been mentioned by Caspers \cite[p. 279]{Caspers2021} and by Skalski and Viselter \cite[p. 62]{Skalski2019}:
\begin{problem}[Abstract version]\label{problem-abstr}
	Is it possible to generalise the construction by Cipriani and Sauvageot of a derivation that is the square root of a generator of a QMS to non-tracial states?
\end{problem}

There are two things that need to be specified to make Problem \ref{problem-abstr} concrete. Let $\rho$ be a state on $\A$. There are many natural inner products on $\A$ based on $\rho$ depending on the embedding of $\A$ in its $L^2$ space, as was observed by Kosaki \cite{Kosaki1984}, and we need to choose which one we want to use. One class of inner products is conveniently described by Carlen and Maas \cite{CARLEN20171810} as the inner products given by 
\begin{equation*}
	\sip{A,B}_s=\rho(\sigma_{si}^{\rho}(B^*)A),
\end{equation*}
for $s\in[0,1]$, with $\sigma_t^{\rho}$ the modular automorphism group with respect to $\rho$. $\sip{\cdot,\cdot}_0$ and $\sip{\cdot,\cdot}_{\frac{1}{2}}$ are called the GNS or KMS inner product, respecively. We also need to decide which requirements we put on the codomain $H$ of the derivation. It is clear that $H$ needs to be an $\A-\A$ bimodule. In the case that we do not impose stronger conditions on $H$, Carlen and Maas have shown that when $\A$ is finite-dimensional and $(\mathcal{P}_t)_{t\geq 0}$ is a QMS with detailed balance, i.e. it is symmetric with respect to the GNS inner product, then there exists a derivation $\delta$ from $\A$ with the KMS inner product to an $\A-\A$ bimodule $H$ with an inner product such that the generator of $(\mathcal{P}_t)_{t\geq0}$ is given by $\delta^*\circ\delta$ \cite{Carlen2020}. However, this bimodule is not a $*$-bimodule, which is desirable when one wants to go to infinite dimensions. Therefore, we want to require that the codomain of the derivation is in fact a Hilbert bimodule, as was the case in the construction of Cipriani and Sauvageot, leading us to the following concrete version of Problem \ref{problem-abstr}:
\begin{problem}[Concrete version]\label{problem-concr}
	Let $s\in[0,1]$. For a unital C*-algebra $\A$ with faithful, lower semicontinuous state $\rho$ and a QMS $(\mathcal{P}_t)_{t\geq 0}$ on $\A$ with generator $L$, symmetric with respect to the GNS inner product, do there always exist a Hilbert $\A-\A$ bimodule $H$ and a densely defined derivation $\delta:\A\rightarrow H$ such that $L^{(2)}=\delta^*\circ\overline{\delta}$ when we consider $\delta$ as an operator on $L^2(\A,\sip{\cdot,\cdot}_s)$?
\end{problem}
 
Here, we will provide a negative answer to Problem \ref{problem-concr} for $s=0$ and $s=\frac{1}{2}$. Additionally, the method we use for $s=0$ and $s=\frac{1}{2}$ can be applied directly to the other cases, and we hypothesise that the answer is negative for all $s\in[0,1]$.

We conclude the introduction by giving a short overview of the proof. We will first show that we only have to consider Hilbert spaces of a concrete form in Section \ref{sec-gen-form}. Next, we will prove Theorem \ref{thm-sys-lin-eqs}, which states that for finite-dimensional unital C*-algebras the question whether such a derivation exists can be transformed to the solvability of a system of linear equations, with the additional requirement that the resulting solution matrix is positive. Together this gives a method to check for any finite-dimensional C*-algebra $\A$ with state $\rho$, inner product $\sip{\cdot,\cdot}_s$ and generator of a QMS $L$ whether there exists a derivation $\delta$ to a Hilbert bimodule such that $L^{(2)}=\delta^*\circ\delta$. We conclude by applying these methods to GNS-self-adjoint generators of QMSs on $\M{2}{C}$ and $\M{3}{C}$ with the GNS or KMS inner product in Section \ref{sec-generators}. We show that these derivations sometimes, but not always, exist by using Mathematica to solve the systems of linear equations that appear. This shows that Problem \ref{problem-concr} has a negative answer in the $\M{3}{C}$ case. A more detailed analysis of the $\M{3}{C}$ case with the GNS inner product shows that states and corresponding generators for which such a derivation exist are rare. This is made precise in Example \ref{examp-detailed}.

After completion of this work, a preprint by Wirth became available, which shows that the construction of Cipriani and Sauvageot can be generalised if one generalises the derivation to a so called \emph{twisted derivation} \cite{Wirth2022Diff}. Our results can be seen as proof that this generalisation to twisted derivations cannot be avoided and suggest that the result by Wirth is the best one can hope to achieve.
\\
\\
\emph{Acknowledgements:} The author would like to thank Martijn Caspers for carefully reading this manuscript and for helping place this work in context.

\section{Preliminaries}
We start this work by agreeing on some nomenclature. First, we will assume that all of our sesquilinear forms are linear in the first coordinate and antilinear in the second.  Next, we will precisely define the concepts \emph{bimodule}, \emph{$*$-bimodule} and \emph{Hilbert bimodule}, since their distinction is essential to this work.

\begin{definition}
	Let $\A$ be an algebra. An $\A-\A$ bimodule is a triple $(H,L,R)$ of a vector space $H$ and two algebra homomorphisms $L$ and $R$ from $\A$ and $\A^{\circ}$, respectively, to the linear maps on $H$ such that $L(\A)$ and $R(\A)$ commute. We will denote the left and right action as left and right multiplication, respectively, so for all $A\in\A$ and $x\in H$ we have $Ax=L(A)x$ and $xA=R(A)x$.
\end{definition}
\begin{definition}
	Let $\A$ be a $*$-algebra. An $\A-\A$ $*$-bimodule is an $\A-\A$ bimodule $H$ with an positive sesquilinear form $\sip{\cdot,\cdot}$ such that for all $x,y\in H$ and $A\in\A$ we have $\sip{Ax,y}=\sip{x,A^*y}$ and $\sip{xA,y}=\sip{x,yA^*}$ and $L(A)$ and $R(A)$ are bounded. If $H$ is a Hilbert space, then we call this a Hilbert $\A-\A$ bimodule.
\end{definition}
We will often just call something a ($*$-)bimodule if the algebra is clear. Subsequently, we include the definition of a derivation.
\begin{definition}
 Let $\A$ be an algebra and let $H$ be an $\A-\A$ bimodule. A linear map $\delta:\A\rightarrow H$ is a derivation if it satisfies the product rule $\delta(AB)=A\delta(B)+\delta(A)B$ for all $A,B\in\A$. 
\end{definition}

Lastly, we will give a precise definition of a quantum Markov semigroup on a unital C*-algebra $\A$ and its generator, and we will define when we call a QMS symmetric. We finish by defining the extension of a $\rho$-symmetric QMS and its generator to $L^2(\A,\sip{\cdot,\cdot}_s)$. 
\begin{definition}
	A quantum Markov semigroup (QMS) on a unital C*-algebra $\A$ is a strongly continuous one-parameter semigroup of bounded linear maps $(\mathcal{P}_t)_{t\geq 0}$ on $\A$ such that for each $t\geq 0$ we have that $\mathcal{P}_t$ is completely positive and $\mathcal{P}_t(1)=1$. If $(\mathcal{P}_t)_{t\geq 0}$ is a QMS, then the closed, densely defined operator $L$ on $\A$ satisfying $\mathcal{P}_t=e^{-tL}$ is called the generator of this QMS. 
\end{definition}
\begin{definition}
Let $(\mathcal{P}_t)_{t\geq 0}$ be a QMS on a C*-algebra $\A$ with inner product $\sip{\cdot,\cdot}$. We call $(\mathcal{P}_t)_{t\geq 0}$ symmetric if each $\mathcal{P}_t$ is symmetric with respect to $\sip{\cdot,\cdot}$. If $\rho$ is a state on $\A$, then we call $(\mathcal{P}_t)_{t\geq 0}$ symmetric with respect to $\rho$ or $\rho$-symmetric if it is symmetric with respect to the GNS inner product for $\rho$.
\end{definition}

\begin{remark}
	The GNS inner product for $\rho$ in the previous definition corresponds to the embedding $x\mapsto xd_{\rho}^{\frac{1}{2}}$ of $\A$ into $L^2(\A,\sip{\cdot,\cdot}_0)$ \cite{Kossakowski1977}, where $d_{\rho}$ is the $L^1$ element corresponding to $\rho$ \cite{Terp1981}. This is different from the KMS embedding $x\mapsto \sigma_{-\frac{i}{4}}(x)d_{\rho}^{\frac{1}{2}}$, which, Cipriani explains, can be seen as the most natural one \cite{Cipriani1997}. However, symmetry, i.e. self-adjointness as an operator on the $L^2$ space, imposes a weaker condition on the QMS if one uses the KMS embedding, since GNS symmetry implies that the QMS commutes with the modular automorphism group \cite[Proposition 2.2]{Wirth2022Christ} and, consequently, that the QMS is KMS symmetric. Here we want to prove a no-go theorem, and therefore we want to use the inner product that imposes the strongest condition on the QMS, and then show that one still cannot always find the desired derivation.
\end{remark}

\begin{definition}
	The extension of a $\rho$-symmetric QMS $\mathcal{P}_t$ to $L^2(\A,\sip{\cdot,\cdot}_s)$ is given by
	\begin{equation*}
		\mathcal{P}_t^{(2)}(d_{\rho}^{\frac{s}{2}}xd_{\rho}^{\frac{1-s}{2}})=d_{\rho}^{\frac{s}{2}}\mathcal{P}_t(x)d_{\rho}^{\frac{1-s}{2}}.
	\end{equation*}
	This extension works for all $s\in[0,1]$ \cite[Section 5]{Haagerup2010}. The $L^2$-generator $L^{(2)}$ is defined by
	\begin{equation*}
		L^{(2)}\eta=\lim_{t\rightarrow 0}\frac{\eta-\mathcal{P}_t^{(2)}(\eta)}{t},
	\end{equation*}
	and the domain is given by the $\eta$ for which this limit exists \cite[Section 4]{Cipriani1997}.
\end{definition}
\begin{remark}
	If $\mathcal{P}_t$ is a $\rho$-symmetric QMS with generator $L$, then $\mathcal{P}_t^{(2)}=e^{-tL^{(2)}}$ \cite{Cipriani1997}, so $L^{(2)}$ is an extension of $L$. 
\end{remark}

\begin{remark}
	If $\A$ is finite dimensional, we will just write $\mathcal{P}_t$ and $L$ in all situations, since we do not need extensions in that case.
\end{remark}

\section{The general form of a derivation} \label{sec-gen-form}
In this section we aim to show that any derivation from a $*$-algebra $\A$ to a Hilbert space can be viewed as a derivation from $\A$ to $\A\otimes \A$ with an appropriate left-multiplication. This bimodule structure on $\A\otimes\A$ and derivation were also used in the work of Cipriani and Sauvageot \cite{CIPRIANI200378}. More precisely, we will prove the following theorem.
\begin{theorem}\label{thm-gen-deriv-form}
	Let $\A$ be a $*$-algebra and $H$ a Hilbert $\A-\A$ bimodule with nondegenerate right action. Suppose that $\delta: \A\rightarrow H$ is a derivation. Then there exist a positive sesquilinear form $\sip{\cdot,\cdot}$ and a $*$-bimodule structure on $\A\otimes \A$, an isometric bimodule homomorphism $\phi:\A\otimes \A\rightarrow H$, which extends to $\overline{\A\otimes\A/\sip{\cdot,\cdot}}$, and a derivation $\partial: \A\rightarrow \overline{\A\otimes\A/\sip{\cdot,\cdot}}$ such that $\delta=\phi\circ \partial$ and $\A\delta(\A)\A\subset\phi(\A\otimes\A)$. The bimodule structure on $\A\otimes \A$ is given by
	\begin{equation*}
		A(B\otimes C)D=AB\otimes CD-A\otimes BCD
	\end{equation*}
	for all $A,B,C,D\in\A$. If $\A$ is unital, then $\partial$ maps via $\A\otimes\A$ and is given by
	\begin{equation*}
		\partial(A)=[A\otimes 1]\in \A\otimes\A/\sip{\cdot,\cdot},
	\end{equation*}
	where $[\cdot]$ denotes the equivalence class in the quotient.
\end{theorem}
\begin{remark}
	Note that for unital $*$-algebras $\A$ the nondegeneracy condition is equivalent to the right action being unital, as can be seen by a short calculation.
\end{remark}
\begin{proof}
	Let $H_0\subset H$ be the linear subspace generated by $\delta(\A)\A$. In other words, $H_0$ is given by
	\begin{equation*}\label{H0-2}
		H_0=\{\sum_{i=1}^n\delta(A_i)B_i|n\in \N,A_i,B_i\in \A, i\leq n\}.
	\end{equation*}
	Due to the Leibniz rule, we have that 
	\begin{equation*}
		\delta(AB)C-\delta(A)BC=A\delta(B)C=A\delta(BC)-AB\delta(C),
	\end{equation*}
	showing that $H_0$ contains $\A\delta(\A)\A$. Consequently, $H_0$ is an $\A-\A$ bimodule.
	
	First, we define the bimodule $K$ as the vector space $\A\otimes \A$ with the left and right multiplication given by
	\begin{equation*}
		A(B\otimes C)=AB\otimes C-A\otimes BC\text{ and } (A\otimes B)C=A\otimes BC
	\end{equation*}
	for all $A,B,C\in \A$, respectively. The right multiplication is clearly associative, and for the left multiplication the computation
	\begin{align*}
		A(B(C\otimes D))=A(BC\otimes D-B\otimes CD)=ABC\otimes D&-A\otimes BCD-AB\otimes CD+A\otimes BCD\nonumber\\
		&=ABC\otimes D-AB\otimes CD=(AB)(C\otimes D)
	\end{align*}
	shows that it is indeed associative. This allows us to define the linear map $\phi: K\rightarrow H_0$ by $\phi(A\otimes B)=\delta(A)B$. This is clearly a right $\A$-module homomorphism, and the fact that it is a left $\A$-module homomorphism follows from
	\begin{equation*}
		\phi(A(B\otimes C))=\phi(AB\otimes C-A\otimes BC)=\delta(AB)C-\delta(A)BC=A\delta(B)C=A\phi(B\otimes C).
	\end{equation*}
	$\phi$ is surjective by definition of $H_0$. Since $\A\delta(\A)\A\subset H_0$, we know that $\A\delta(\A)\A\subset \phi(\A\otimes \A)$.
	
	The next step is to define a positive sesquilinear form $\sip{\cdot,\cdot}_K$ on $K$, which we do by setting
	\begin{equation*}
		\sip{A\otimes B,C\otimes D}_K=\sip{\delta(A)B,\delta(C)D}_H=\sip{\phi(A\otimes B),\phi(C\otimes D)}_H.
	\end{equation*}
	Because $\phi$ is a bimodule homomorphism, it also follows that the left and right multiplications on $K$ are $*$-homomorphisms. It is immediately clear that $\phi$ is isometric by the definition of the positive sesquilinear form on $K$. Because $\phi$ is isometric, we find that $\ker(\phi)=\{x\in K|\sip{x,x}_K=0\}$. Since $\phi$ is a bimodule homomorphism, we know that this kernel is a subbimodule. Therefore, $K'=K/\sip{\cdot,\cdot}$ is a bimodule and $\phi$ is an isometric bimodule isomorphism from $K'$ to $H_0$. Consequently, we can extend $\phi$ from the completion of $K'$, $\overline{K'}$, to $\overline{H_0}$. Because $L(A)$ and $R(A)$ are bounded for all $A\in \A$, we know that $\overline{K'}$ and $\overline{H_0}$ are Hilbert bimodules. By the same reasoning we see that $K$ is a $*$-bimodule.
	
	Lastly, we need to define the derivation $\partial:\A\rightarrow \overline{\A\otimes\A/\sip{\cdot,\cdot}}$. First suppose that $\A$ is unital. In this case we will define $\partial: \A\rightarrow\A\otimes\A/\sip{\cdot,\cdot}$ by $\partial(A)=[A\otimes 1]$, which then automatically also maps into $\overline{\A\otimes\A/\sip{\cdot,\cdot}}$. This is a derivation, because
	\begin{equation*}
		A(B\otimes 1)+(A\otimes 1)B=AB\otimes 1-A\otimes B+A\otimes B=AB\otimes 1
	\end{equation*}
	and taking the equivalence class on both sides gives the Leibniz rule. Furthermore, we see that $\delta=\phi\circ\partial$ indeed holds. This concludes the proof for the unital case.
	
	Now we turn to the non-unital case. To define our derivation, we look at the densely defined linear functional $l_A$ on $\overline{K'}$ given by
	\begin{equation*}
		l_A([B\otimes C])=\sip{\delta(A),\delta(B)C}_H
	\end{equation*}
	for all $B,C\in\A$. We know that $\norm{l_A}\leq \sip{\delta(A),\delta(A)}_H^{\frac{1}{2}}$ by the Cauchy-Schwarz inequality, so $l_A$ can be extended to a bounded linear functional. Consequently, there exists a unique $x\in \overline{K'}$ such that 
	\begin{equation*}
		l_A([B\otimes C])=\sip{x,[B\otimes C]}_{K'}
	\end{equation*}
	holds for all $B,C\in\A$, since $\overline{K'}$ is a Hilbert space. This allows us to define $\partial(A)$ for each $A\in\A$ as the unique element in $\overline{K'}$ such that
	\begin{equation*}
		\sip{\partial(A),[B\otimes C]}_{K'}=\sip{\delta(A),\delta(B)C}_H
	\end{equation*}
	holds for all $B,C\in\A$. We can now do the required computation to show that $\partial$ is in fact a derivation. Here we will use that both $H$ and $\overline{K'}$ are $*$-bimodules and that $\phi$ is a bimodule homomorphism. Let $A,B,C,D\in\A$ be arbitrary. Then we have
	\begin{align*}
		\sip{A\partial(B)+\partial(A)B,[C\otimes D]}_{K'}&=\sip{A\partial(B),[C\otimes D]}_{K'}+\sip{\partial(A)B,[C\otimes D]}_{K'}\\
		&=\sip{\partial(B),A^*([C\otimes D])}_{K'}+\sip{\partial(A),[C\otimes DB^*]}_{K'}\\
		&=\sip{\delta(B),\phi(A^*([C\otimes D]))}_H+\sip{\delta(A),\phi([C\otimes DB^*])}_H\\
		&=\sip{\delta(B),A^*\phi([C\otimes D])}_H+\sip{\delta(A),\phi([C\otimes D])B^*}_H\\
		&=\sip{A\delta(B),\phi([C\otimes D])}_H+\sip{\delta(A)B,\phi([C\otimes D])}_H\\
		&=\sip{\delta(AB),\phi([C\otimes D])}_H\\
		&=\sip{\partial(AB),[C\otimes D]}_{K'},
	\end{align*}
	which shows that $\partial(AB)=\partial(A)B+A\partial(B)$. So we have shown that $\partial$ is a derivation. All that remains is to prove that $\delta=\phi\circ\partial$. For all $A,B,C\in\A$ we have
	\begin{equation*}
		\sip{\phi(\partial(A)),\delta(B)C}_H=\sip{\phi(\partial(A)),\phi([B\otimes C])}_H=\sip{\partial(A),[B\otimes C]}_{K'}=\sip{\delta(A),\delta(B)C}_H.
	\end{equation*}
	Consequently, for all $A,B,C\in\A$ we see that 
	\begin{equation*}
		\sip{\phi(\partial(A))-\delta(A),\delta(B)C}_H=0.
	\end{equation*}
	But then we also have for all $A,B,C,D\in \A$ that 
	\begin{equation*}
		\sip{(\phi(\partial(A))-\delta(A))D^*,\delta(B)C}_H=\sip{\phi(\partial(A))-\delta(A),\delta(B)CD}_H=0,
	\end{equation*}
	which shows that for any $x\in \overline{H_0}$ and $A,D\in\A$ we have
	\begin{equation*}
		\sip{(\phi(\partial(A))-\delta(A))D,x}_H=0.
	\end{equation*}
	Since $\delta(A)D$ is contained in $\overline{H_0}$ for all $A,D\in\A$, we have
	\begin{equation*}
		\sip{(\phi(\partial(A))-\delta(A))D,(\phi(\partial(A))-\delta(A))D}_H=0.
	\end{equation*}
	Therefore we can conclude that $(\phi(\partial(A))-\delta(A))D=0$ for all $A,D\in\A$. But we know that the right action is nondegenerate, so this proves that $\phi(\partial(A))=\delta(A)$, which is what we wanted to show.
\end{proof}
\begin{remark}
	From now on we will refer to $\A\otimes \A$ with the bimodule structure given by
	\begin{equation*}
		A(B\otimes C)D=AB\otimes CD-A\otimes BCD
	\end{equation*}
	for all $A,B,C,D\in\A$ as the canonical bimodule $\A\otimes\A$.
\end{remark}
\begin{corollary}\label{corr-gen-deriv-f}
Let $\A$ be a unital $*$-algebra and $f:\A\times\A\rightarrow \C$ a sesquilinear map. There exists a derivation $\delta: \A\rightarrow H$ from $\A$ to a Hilbert bimodule $H$ with unital right action such that $\sip{\delta(A),\delta(B)}=f(A,B)$ for all $A,B\in\A$ if and only if there exist a positive sesquilinear form $\sip{\cdot,\cdot}_{\A\otimes\A}$ on the canonical bimodule $\A\otimes \A$ such that $\A\otimes \A$ is a $*$-bimodule and $\sip{A\otimes 1,B\otimes 1}_{\A\otimes\A}=f(A,B)$ holds for all $A,B\in\A$.
\end{corollary}
\begin{proof}
	The only if direction follows from Theorem \ref{thm-gen-deriv-form}. The converse direction follows by completing $\A\otimes \A/\sip{\cdot,\cdot}$. This gives a bimodule because $\{x\in \A\otimes \A|\sip{x,x}=0\}=\{x\in\A\otimes\A|\sip{x,y}=0\ \forall y\in \A\otimes \A\}$ is a subbimodule and the left and right actions extend to the completion because they are bounded. 
\end{proof}

\section{Existence of derivation as solution of system of linear equations}
Our interest in the existence of derivations comes from the question whether a derivation exists that is the square root of a generator $L$ of a quantum Markov semigroup. This means that we are looking for a derivation $\delta$ from a $*$-algebra $\A$ to a Hilbert bimodule $H$ such that $\sip{\delta(A),\delta(B)}_H=\sip{L^{\frac{1}{2}}(A),L^{\frac{1}{2}}(B)}_{\A}$ for all $A,B\in\A$. Now Theorem \ref{thm-gen-deriv-form} allows us to consider only $*$-bimodules of a specific form. This will be essential to formulate a procedure to either find such a derivation or show that it does not exist. 
\begin{theorem}\label{thm-sys-lin-eqs}
	Let $\A$ be a unital finite-dimensional $*$-algebra with dimension $m$ and let $f:\A\times\A\rightarrow \C$ be a sesquilinear map. Fix an isomorphism $\psi:\A\otimes\A\rightarrow \C^{m^2}$ and a basis $\Q=\{Q_1,\dots Q_m\}$ of $\A$. Then the following are equivalent:
	\begin{enumerate}
		\item There exists a positive $m^2\times m^2$ matrix $X$ that is a solution to the system of equations given by
		\begin{align}
			&\{\psi(Q_{i_4}^*\otimes Q_{i_5}^*)^*X\psi(Q_{i_1}(Q_{i_2}\otimes Q_{i_3}))-\psi(Q_{i_1}^*(Q_{i_4}^*\otimes Q_{i_5}^*))^*X\psi(Q_{i_2}\otimes Q_{i_3})=0|1\leq i_1,\dots,i_5\leq m\}\nonumber\\
			&\cup \{\psi(Q_{i_4}^*\otimes Q_{i_5}^*)^*X\psi((Q_{i_2}\otimes Q_{i_3})Q_{i_1})-\psi((Q_{i_4}^*\otimes Q_{i_5}^*)Q_{i_1}^*)^*X\psi(Q_{i_2}\otimes Q_{i_3})=0|1\leq i_1,\dots,i_5\leq m\}\nonumber\\
			&\cup \{\psi(Q_{i_2}^*\otimes 1)^*X\psi(Q_{i_1}\otimes 1)-f(Q_{i_1},Q_{i_2}^*)=0|1\leq i_1,i_2\leq m\}\label{eq-sys-lin-eqs},
		\end{align}
		which is linear in $X$.
		\item There exists a derivation $\delta$, given by $\delta(A)=A\otimes 1$, to the canonical bimodule $\A\otimes \A$ and a positive sesquilinear form $\sip{\cdot,\cdot}_{\A\otimes\A}$ on $\A\otimes \A$ such that $\A\otimes\A$ with this positive sesquilinear form is a $*$-bimodule and for all $A,B\in\A$ we have $\sip{\delta(A),\delta(B)}_{\A\otimes\A}=f(A,B)$.
	\end{enumerate}
\end{theorem}
\begin{proof}
	By Theorem \ref{thm-gen-deriv-form} we know that we can assume without loss of generality that our derivation $\delta$ will have the form $\delta(A)=A\otimes 1$ for all $A\in\A$. This means that we already have a derivation from $\A$ to $\A\otimes \A$, so the positive sesquilinear form $\sip{\cdot,\cdot}_{\A\otimes\A}$ is all that is left to choose. The idea is to write down linear equations that capture the properties that we require of the positive sesquilinear form. These properties are:
	\begin{enumerate}[(i)]
		\item The left multiplication is a $*$-homomorphism, i.e. $\sip{A(B\otimes C),D\otimes E}_{\A\otimes\A}=\sip{B\otimes C,A^*(D\otimes E)}_{\A\otimes\A}$ for all $A,B,C,D,E\in\A$.
		\item The right multiplication is a $*$-homomorphism, i.e. $\sip{(B\otimes C)A,D\otimes E}_{\A\otimes\A}=\sip{B\otimes C,(D\otimes E)A^*}_{\A\otimes\A}$ for all $A,B,C,D,E\in\A$.
		\item The positive sesquilinear form takes the required values on the range of $\delta$, i.e. $\sip{\delta(A),\delta(B)}_{\A\otimes\A}=f(A,B)$ for all $A,B\in\A$.
	\end{enumerate}
	Any positive sesquilinear form $\sip{\cdot,\cdot}$ on $\C^{m^2}$ can be represented by a positive $m^2\times m^2$ matrix $X$ by the relation
	\begin{equation*}
		\sip{v,w}=w^*Xv,
	\end{equation*}
	if we view $v,w\in\C^{m^2}$ as column vectors. We can therefore try to describe the positive sesquilinear form by finding linear equations for the matrix $X$ such that for all $A,B,C,D\in\A$ we have
	\begin{equation*}
		\sip{A\otimes B,C\otimes D}_{\A\otimes\A}=\psi(C\otimes D)^*X\psi(A\otimes B).
	\end{equation*}
	
	We will start with the left multiplication. Consider the map $\phi_{\mathrm{l}}:\A^5\rightarrow\C$ given by
	\begin{equation*}
		\phi_{\mathrm{l}}(A,B,C,D,E)=\sip{A(B\otimes C),D^*\otimes E^*}_{\A\otimes\A}-\sip{B\otimes C,A^*(D^*\otimes E^*)}_{\A\otimes\A}.
	\end{equation*}
	We know that the left multiplication is a $*$-homomorphism if and only if $\phi_{\mathrm{l}}$ is the zero funciton. In fact, since $\phi_{\mathrm{l}}$ is linear in each of the five coordinates, we know that $\phi_{\mathrm{l}}$ is the zero function if and only if $\phi_{\mathrm{l}}(A,B,C,D,E)=0$ for all combinations $A,B,C,D,E\in \Q$. Because $\A$ is finite-dimensional, this condition only imposes a finite number of conditions on the positive sesquilinear form. Converting $\phi_{\mathrm{l}}(A,B,C,D,E)=0$ to an equation containing $X$, we see that we obtain the set of equations
	\begin{equation*}
		R_{\mathrm{l}}=\{\psi(Q_{i_4}^*\otimes Q_{i_5}^*)^*X\psi(Q_{i_1}(Q_{i_2}\otimes Q_{i_3}))-\psi(Q_{i_1}^*(Q_{i_4}^*\otimes Q_{i_5}^*))^*X\psi(Q_{i_2}\otimes Q_{i_3})=0|1\leq i_1,\dots,i_5\leq m\}.
	\end{equation*}
	So we know that the positive sesquilinear form given by $X$ turns the left multiplication into a $*$-homomorphism if and only if all equations in $R_{\mathrm{l}}$ hold.	
	Analogously, we find that the positive sesquilinear form by $X$ turns the right multiplication into a $*$-homomorphism if and only if all equations in
	\begin{equation*}
		R_{\mathrm{r}}=\{\psi(Q_{i_4}^*\otimes Q_{i_5}^*)^*X\psi((Q_{i_2}\otimes Q_{i_3})Q_{i_1})-\psi((Q_{i_4}^*\otimes Q_{i_5}^*)Q_{i_1}^*)^*X\psi(Q_{i_2}\otimes Q_{i_3})=0|1\leq i_1,\dots,i_5\leq m\}
	\end{equation*}
	hold.
	
	This leaves property (iii), for which we look at the map $\phi_{\mathrm{f}}:\A^2\rightarrow \C$ given by
	\begin{equation*}
		\phi_{\mathrm{f}}(A,B)=\sip{A\otimes 1,B^*\otimes 1}-f(A,B^*).
	\end{equation*}
	This is once again a linear map in both coordinates, and property (iii) holds if and only if $\phi_{\mathrm{f}}$ is the zero function. By the same reasoning as before, we find that property (iii) holds if and only if $\phi_{\mathrm{f}}(A,B)=0$ for all pairs $A,B\in\Q$. Translated to $X$ this means that property (iii) holds if and only if all equations in
	\begin{equation*}
		R_{\mathrm{f}}=\{\psi(Q_{i_2}^*\otimes 1)^*X\psi(Q_{i_1}\otimes 1)-f(Q_{i_1},Q_{i_2}^*)=0|1\leq i_1,i_2\leq m\}
	\end{equation*}
	hold.
	
	Combining all these equations, we find that there exists a positive sesquilinear form on $\A\otimes \A$ satisfying (i), (ii) and (iii) if and only if there exists a positive matrix $X$ that is a simultaneous solution of all equations in $R_{\mathrm{l}}$, $R_{\mathrm{r}}$ and $R_{\mathrm{f}}$.
\end{proof}
\begin{corollary} \label{corr-lin-sys-eqs}
	Let $\A$ be a unital finite-dimensional $*$-algebra with dimension $m$ and let $f:\A\times\A\rightarrow \C$ be a sesquilinear map. Fix an isomorphism $\psi:\A\otimes\A\rightarrow \C^{m^2}$ and a basis $\Q=\{Q_1,\dots Q_m\}$ of $\A$. If condition 1 of Theorem \ref{thm-sys-lin-eqs} is not satisfied, then no Hilbert bimodule $H$ with unital right action and derivation $\delta:\A\rightarrow H$ exist such that $\sip{\delta(A),\delta(B)}_H=f(A,B)$.
\end{corollary}
\begin{proof}
	This is immediate from Corollary \ref{corr-gen-deriv-f} and Theorem \ref{thm-sys-lin-eqs}.
\end{proof}

\section{Existence of a square root of a generator of a QMS} \label{sec-generators}
The developed theory will allow us to tackle Problem \ref{problem-concr}. The formulation in the introduction requires the use of Tomita-Takesaki theory to understand the modular automorphism group, which can be found in \cite{Takesaki2003II}. However, if $\A$ is a finite-dimensional C*-algebra, the modular automorphism group for a faithful state $\rho$ is given by 
\begin{equation*}
\sigma_t^{\rho}(A)=d_{\rho}^{it}Ad_{\rho}^{-it}
\end{equation*}
for all $A\in\A,t\in \C$, where $d_{\rho}$ is the density matrix corresponding to $\rho$. Consequently, the inner products $\sip{\cdot,\cdot}_s$ on $\A$ are then given by
\begin{equation*}
	\sip{A,B}_s=\tau(d_{\rho}^{1-s}B^*d_{\rho}^sA),
\end{equation*}
where $\tau$ is the normalised trace on $\A$. For a fixed $s\in[0,1]$ and using the notation above, we then want to know whether, given a finite-dimensional C*-algebra $\A$, a faithful state $\rho$ on $\A$ and a generator $L$ of a QMS such that for all $A,B\in\A$: $\sip{L(A),B}_0=\sip{A,L(B)}_0$, we can find a $\A-\A$ $*$-bimodule $H$ with positive sesquilinear form $\sip{\cdot,\cdot}_H$ and a derivation $\delta:\A\rightarrow H$ such that $\tau(d_{\rho}^{1-s}B^*d_{\rho}^sL(A))=\sip{\delta(A),\delta(B)}_H$ for all $A,B\in\A$.

Using concrete examples for the case that $\rho$ is not tracial, we will show that this is sometimes, but not always, possible for both $s=0$ and $s=\frac{1}{2}$. To obtain these concrete generators, we need part of a theorem stated by Carlen and Maas \cite[Theorem 3.1]{CARLEN20171810} and originally proven by Alicki \cite{Alicki1976}. This theorem was based on the description of generators of QMSs by Lindblad \cite{Lindblad1976} and more directly on the work of Gorini, Kossakowski and Sudarshan \cite{Gorini1976}.
\begin{theorem} \label{thm-CM}
	Let $\rho$ be a faithful state on $\M{n}{C}$. Let $\mathcal{J}$ be a finite index set and let $\{V_j\}_{j\in\mathcal{J}}\subset\M{n}{C}$ and $\{\omega_j\}_{j\in\mathcal{J}}\subset\R$ be such that
	\begin{equation} \label{eq-CM-cond}
		\{V_j\}_{j\in\mathcal{J}}=\{V_j^*\}_{j\in\mathcal{J}}\text{ and } \sigma_{-i}^{\rho}(V_j)=e^{-\omega_j}V_j.
	\end{equation}
	Then the operator $L:\M{n}{C}\rightarrow\M{n}{C}$, given by
	\begin{equation*}
		L(A)=-\sum_{j\in\mathcal{J}}e^{-\frac{\omega_j}{2}}\left(V_j^*[A,V_j]+[V_j^*,A]V_j\right),
	\end{equation*}
	is a generator of a QMS that is $\rho$-symmetric. Conversely, any generator of a QMS on $\M{n}{C}$ that is $\rho$-symmetric is of the above form.
\end{theorem}
We can now give the examples that prove the claim that the desired derivation sometimes, but not always, exists. Because of Theorem \ref{thm-CM}, we can construct a real generator $L$ of a QMS that is self-adjoint with respect to the GNS inner product by choosing an $n\in\N$, a faithful state $\rho$ on $\M{n}{C}$, an index set $\mathcal{J}$ and two sets $\{V_j\}_{j\in\mathcal{J}}\subset\M{n}{C}$ and $\{\omega_j\}_{j\in\mathcal{J}}\subset\R$ that satisfy Equation \ref{eq-CM-cond}. To use Corollary \ref{corr-lin-sys-eqs} or Theorem \ref{thm-sys-lin-eqs} we also need to choose an isomorphism $\psi:\M{n}{C}\otimes \M{n}{C}\rightarrow \C^{n^4}$ and a basis $\Q$ for $\M{n}{C}$. For $f:\A\times\A\rightarrow\C$ we can then pick $f(A,B)=\rho(B^*L(A))$ if we consider the GNS inner product on $\M{n}{C}$ or $f(A,B)=\rho(\sigma_{\frac{i}{2}}^{\rho}(b^*)a)$ if we look at the KMS inner product. We then find the system of linear equations \ref{eq-sys-lin-eqs}, which we solve using Mathematica. All computations using Mathematica are symbolic. We will numerically evaluate some eigenvalues, but solely to provide some feeling for the results. The code that was used can be found at \url{www.doi.org/10.4121/19323878}.

\begin{example}[Generator with derivation as square root]\label{examp-2x2}
	For this example we pick
	\begin{align*}
		n=2,\ \mathcal{J}=\{1,2\},\ e^{-\omega_1}=\frac{\pi+1}{\pi-1}, e^{-\omega_2}=\frac{\pi-1}{\pi+1},\ \psi(E_{ij}&\otimes E_{kl})=e_{8i+4k+2j+l-14},\ \Q=\{E_{ij}|1\leq i,j\leq 2\},\\
		\rho(A)=\frac{1}{2}\mathrm{Tr}\left(\begin{pmatrix}
		1+\frac{1}{\pi}&0\\0&1-\frac{1}{\pi}
		\end{pmatrix}A\right),&\ V_1=\begin{pmatrix}
			0&1\\0&0
		\end{pmatrix}\text{ and }\ V_2=V_1^*
	\end{align*}
	and we consider the GNS inner product on $\M{2}{C}$. If we pick $f:\A\times\A\rightarrow\C: f(A,B)=\rho(B^*L(A))$, then we can concretely check statement (1) of Theorem \ref{thm-sys-lin-eqs} using Mathematica. We obtain a self-adjoint matrix $X$ that is a solution of the system of equations and whose eigenvalues can be expressed algebraically in terms of $\pi$ and are all non-negative. They are approximately given by
	\begin{equation*}
		\{1.96, 1.96, 0.96, 0.96, 0.64, 0.64, 0.31, 
0.31, 0.07, 0.07, 0.07, 0.07, 0, 0, 0, 0\},
	\end{equation*}
	showing that $X$ is indeed positive. By Theorem \ref{thm-sys-lin-eqs} this shows that there exists a derivation $\delta$ from $\M{2}{C}$ to a $*$-bimodule such that $L=\delta^*\circ\delta$ with respect to the GNS inner product on $\M{2}{C}$.
	
	Alternatively, we can consider the KMS inner product on $\M{2}{C}$ and consequently replace $f$ by $f(A,B)=\rho(\sigma_{\frac{i}{2}}^{\rho}(B^*)L(A))$. In this case we also find a self-adjoint matrix $X$ that solves Equations \ref{eq-sys-lin-eqs} and whose eigenvalues can be expressed algebraically in terms of $\pi$ and are all non-negative. Approximately, the eigenvalues are given by
	\begin{equation*}
		\{1.44, 1.44, 1.44, 1.44, 0.47, 0.47, 0.47, 
0.47, 0.03, 0.03, 0.03, 0.03, 0, 0, 0, 0\}.
	\end{equation*}
	So Theorem \ref{thm-sys-lin-eqs} shows that we can also find a derivation $\delta$ from $\M{2}{C}$ to a $*$-bimodule such that $L=\delta^* \circ \delta$ with respect to the KMS inner product on $\M{2}{C}$.
\end{example}

\begin{example}[Generator without derivation as square root]\label{examp-3x3}
	For this example we pick
	\begin{align*}
		n=3,\ \mathcal{J}=\{1,2\},\ e^{-\omega_1}=\frac{\pi^2}{e^2}, e^{-\omega_2}=\frac{e^2}{\pi^2},\ \psi(E_{ij}&\otimes E_{kl})=e_{27i+9k+3j+l-39},\ \Q=\{E_{ij}|1\leq i,j\leq 3\},\\
		\rho(A)=\frac{1}{1+\pi^2+e^2}\mathrm{Tr}\left(\begin{pmatrix}
		1&0&0\\0&\pi^2&0\\0&0&e^2
		\end{pmatrix}A\right),&\ V_1=\begin{pmatrix}
			0&0&0\\
			0&0&1\\0&0&0
		\end{pmatrix}\text{ and }\ V_2=V_1^*.
	\end{align*}
	If we pick $f:\A\times\A\rightarrow\C: f(A,B)=\rho(B^*L(A))$, then we can concretely check statement (1) of Theorem \ref{thm-sys-lin-eqs} using Mathematica. We find that there does not exist any matrix $X$ that satisfies the system of equations. Using Corollary \ref{corr-lin-sys-eqs} we can now conclude that there do not exist an $\A-\A$ $*$-bimodule $H$ and a derivation $\delta:\A\rightarrow H$ such that $L=\delta^*\circ\delta$ with respect to the GNS inner product on $\A$.
	
	Alternatively, if we consider the KMS inner product on $\A$ and therefore use $f(A,B)=\rho(\sigma_{\frac{i}{2}}^{\rho}(B^*)L(A))$, then we do find a subspace $\mathcal{X}$ of matrices that satsify Equations \ref{eq-sys-lin-eqs}. However, for any matrix $X\in\mathcal{X}$ we have that 
	\begin{equation*}
		\psi(E_{12}\otimes E_{22}+E_{13}\otimes E_{32})^*X\psi(E_{12}\otimes E_{22}+E_{13}\otimes E_{32})=-\frac{e^2+2e(\pi-1)+\pi(\pi-2)}{1+\pi^2+e^2}<0.
	\end{equation*}
	Consequently, we see that $X$ is not positive definite, and therefore that the corresponding sesquilinear form on $\A\otimes \A$ is not positive. Since this holds for any $X\in\mathcal{X}$, we conclude that that there do not exist an $\A-\A$ $*$-bimodule $H$ and a derivation $\delta:\A\rightarrow H$ such that $L=\delta^*\circ\delta$ with respect to the KMS inner product on $\A$.  
\end{example}

\begin{example}\label{examp-detailed}
Let us consider the $\M{3}{C}$ case and a state $\rho$ given by
\begin{equation*}
	\rho(A)=\frac{1}{\lambda_1^2+\lambda_2^2+\lambda_3^2}\Tr\left(\begin{pmatrix}
		\lambda_1^2&0&0\\
		0&\lambda_2^2&0\\
		0&0&\lambda_3^2
	\end{pmatrix}A\right)
\end{equation*}
for some $\lambda_i>0$. Without loss of generality we can assume that $\lambda_1=1$. First we note that choosing $V=\diag(a,b,c)$ with $a,b,c\in\R$ (and $J$ a singleton set) in Theorem \ref{thm-CM} gives the same generator as choosing 
\begin{align*}
	V_1&=\sqrt{\frac{1}{2}((a-b)^2+(a-c)^2-(b-c)^2)}\diag(1,0,0),\\
	V_2&=\sqrt{\frac{1}{2}((a-b)^2+(b-c)^2-(a-c)^2)}\diag(0,1,0) \text{ and }\\
	V_3&=\sqrt{\frac{1}{2}((a-c)^2+(b-c)^2-(a-b)^2)}\diag(0,0,1). 
\end{align*}
Consequently, if all ratios of the $\lambda_i$ are different and unequal to 1, then by Theorem \ref{thm-CM} any generator of a QMS that is self-adjoint with respect to the GNS inner product can be obtained by using the collection of $V_j$'s given by $\{\sqrt{Y_{ij}}E_{ij}|1\leq i,j\leq 3\}$ for some symmetric matrix $Y\in M_{3\times 3}(\mathds{R}_{\geq 0})$. 

Let $\omega_{kl}$ be such that
\begin{equation*}
	e^{-\omega_{kl}}E_{kl}=\sigma_{-i}^{\rho}(E_{kl})=\frac{\lambda_k^2}{\lambda_l^2}E_{kl}.
\end{equation*}
We will now consider a slightly more general situation than the generators of $\rho$-symmetric QMSs. For any symmetric $Y$ in $\M{3}{R}$ define the operator $L_Y:\M{3}{C}\rightarrow \M{3}{C}$ by
\begin{equation*}
	L_Y(A)=-\sum_{i,j=1}^3Y_{ij}e^{-\frac{\omega_{ij}}{2}}\left(E_{ij}^*[A,E_{ij}]+[E_{ij}^*,A]E_{ij}\right).
\end{equation*}
Note that any generator of a QMS that is $\rho$-symmetric can be obtained by choosing the appropriate $Y$. In order to use Theorem \ref{thm-sys-lin-eqs} and Corollary \ref{corr-lin-sys-eqs} we pick the basis $\Q=\{E_{ij}|1\leq i,j\leq 3\}$ of $\M{3}{C}$ and the isomorphism $\psi:\M{3}{C}\otimes \M{3}{C}\rightarrow \C^{3^4}:\psi(E_{ij}\otimes E_{kl})=e_{27i+9k+3j+l-39}$, where the $e_i$ form the standard basis of $\C^{81}$, and we fix the explicit state $\rho$ by choosing $\lambda_2=\pi$ and $\lambda_3=e^{\pi}$. Lastly, we define the function $f_Y:\M{3}{C}\times\M{3}{C}\rightarrow \C:f_Y(A,B)=\rho(B^*L_Y(A))$. Now Equations \ref{eq-sys-lin-eqs} are linear in the pair $(X,Y)$. Consequently, we can solve the system of linear equations to find which conditions on $Y$ must hold to make sure that an $X$ exists that satisfies the equations. Solving this system of linear equations with Mathematica we find that such an $X$ exists if and only if

\begin{align}\label{eq-sol-exists}
\frac{(1-\lambda_3^2-\lambda_2^2)(\lambda_3^2-\lambda_2^2)}{\lambda_3\lambda_2}Y_{23}+\frac{(\lambda_3^2-1-\lambda_2^2)(\lambda_2^2-1)}{\lambda_2}&Y_{12}
+\frac{(\lambda_2^2-1-\lambda_3^2)(1-\lambda_3^2)}{\lambda_3}Y_{13}\nonumber\\
&+(\lambda_3^2-\lambda_2^2)Y_{11}+(1-\lambda_3^2)Y_{22}+(\lambda_2^2-1)Y_{33}=0
\end{align}

with $\lambda_2=\pi$ and $\lambda_3=e^{\pi}$. These calculations have been executed for a specific state $\rho$. However, due to our choice of $\rho$, using the fact that $\pi$ and $e^{\pi}$ are algebraically independent \cite{Nesterenko1996}, we can obtain some more general conclusions.
\begin{prop}\label{prop-P}
	There exists a set $P\subset\R_{> 0}\times \R_{>0}$ such that 
\begin{equation*}
	|\{x\in\R_{>0}:|\{y\in\R_{>0}|(x,y)\in P \text{ or } (y,x)\in P\}|=\infty\}|<\infty
\end{equation*}
and for all $(\lambda_2,\lambda_3)\in(\R_{>0}\times \R_{>0})\backslash P$ and symmetric $Y\in\M{3}{R}$ the following are equivalent:
\begin{enumerate}[(i)]
	\item There exists a solution $X$ of Equations \ref{eq-sys-lin-eqs} with the function $f=f_Y$.
	\item $Y$ satisfies Equation \ref{eq-sol-exists}.
\end{enumerate}
\end{prop}
To prove this, we need a lemma.

\begin{lemma}\label{lem-rational-func}
	Let $m,n\in\N$, $A\in M_{n,m}(\C)$ and $b:\R\times \R\rightarrow \C^n$ be an entrywise rational function. Then there exists a finite family of rational functions $r_1,\dots,r_l$ in two variables such that for all $\lambda_1,\lambda_2\in\R$ the system of linear equations $Ax=b(\lambda_1,\lambda_2)$ with $x\in\C^m$ has a solution if and only if $r_i(\lambda_1,\lambda_2)=0$ for all $1\leq i\leq l$. 
\end{lemma}
\begin{proof}
	Let $k$ be the rank of $A$. By the Rouché-Capelli theorem, we know that our system of linear equations has a solution if and only if $(A|b(\lambda_1,\lambda_2))$ also has rank $k$. Since the rank of $(A|b(\lambda_1,\lambda_2))$ is at least $k$, this is equivalent to the statement that all $(k+1)\times (k+1)$ minors are zero. Since all of the entries of $b(\lambda_1,\lambda_2)$ are rational functions in $\lambda_1$ and $\lambda_2$ and $A$ does not depend on $\lambda_1$ and $\lambda_2$, this means that all of these minors are rational functions in $\lambda_1$ and $\lambda_2$. Choosing these minors as our finite family of rational functions gives the desired result.
\end{proof}

\begin{proof}[Proof of Proposition \ref{prop-P}]
	Let $\Msymm{3}{R}=\{A\in\M{3}{R}|A=A^*\}$ and $\mathcal{M}_{\lambda_2,\lambda_3}\subset \Msymm{3}{R}$ be the subset of matrices satisfying Equation \ref{eq-sol-exists}. We can find a set $\mathcal{T}$ of functions from $\R_{>0}\times \R_{>0}$ to $\Msymm{3}{R}$ such that for all $\lambda_2,\lambda_3\in\R_{>0}$  we have that $\{T(\lambda_2,\lambda_3)|T\in\mathcal{T}\}$ is a basis for $\mathcal{M}_{\lambda_2,\lambda_3}$, and $T_{ij}$ is a rational function of $\lambda_2$ and $\lambda_3$ for all $1\leq i,j\leq 3$ and $T\in\mathcal{T}$. For each $T\in\mathcal{T}$ and $Q_i,Q_j\in\Q$ we have that $f_T(Q_i,Q_j)$ is a rational function in $\lambda_2$ and $\lambda_3$. By Lemma \ref{lem-rational-func} we now know that for all $T\in\mathcal{T}$ there exist a family of rational functions $\mathcal{R}_T$ such that for all $\lambda_2,\lambda_3\in\R_{>0}$: there exists a solution for Equations \ref{eq-sys-lin-eqs} with the function $f=f_{T(\lambda_2,\lambda_3)}$ if and only if $r(\lambda_2,\lambda_3)=0$ for all $r\in\mathcal{R}_T$. Consequently, by linearity, we find for all $\lambda_2,\lambda_3\in\R_{>0}$ that $r(\lambda_2,\lambda_3)=0$ for all $r\in\mathcal{R}=\bigcup_{T\in\mathcal{T}}\mathcal{R}_T$ if and only if for all $Y\in \mathcal{M}_{\lambda_2,\lambda_3}$ the system of linear equations \ref{eq-sys-lin-eqs} with $f=f_Y$ has a solution. We know that this holds for $\lambda_2=\pi$ and $\lambda_3=e^{\pi}$. Since these numbers are algebraically independent, we must have that the numerators of all of these rational functions are equal to the zero function. Therefore we conclude that (ii) implies (i) for all $\lambda_2,\lambda_3\in\R_{>0}$.
	
	For the other direction, we will show that (ii) is a necessary requirement for (i). We can analogously find a finite set of rational functions $\mathcal{R}'$ such that $r(\lambda_2,\lambda_3)=0$ for all $r\in\mathcal{R}'$ if and only if for all $Y\in \Msymm{3}{R}$ a solution exists to Equations \ref{eq-sys-lin-eqs} with $f=f_Y$. We know by the beginning of this example that the latter part is not true for $\lambda_2=\pi$ and $\lambda_3=e^{\pi}$, so there exists an $r\in \mathcal{R}'$ such that $r(\pi,e^{\pi})\neq 0$. Let $P$ be defined by $P=\{(\kappa_1,\kappa_2)\in\R_{>0}\times \R_{>0}|r(\kappa_1,\kappa_2)=0\}$. Then $P$ satisfies
	\begin{equation*}
		|\{x\in\R_{>0}:|\{y\in\R_{>0}|(x,y)\in P \text{ or } (y,x)\in P\}|=\infty\}|<\infty.
	\end{equation*}
	Now let $(\lambda_2,\lambda_3)\in (\R_{>0}\times \R_{>0})\backslash P$ be arbitrary. By definition of $P$ we know that there exists a $Y\in\Msymm{3}{R}$ such that Equations \ref{eq-sys-lin-eqs} for $f=f_Y$ do not have a solution. Since these equations are linear in $Y$, this means that the subspace of $Y\in\Msymm{3}{R}$ for which these equations have a solution has codimension greater or equal to 1 in $\Msymm{3}{R}$. But we already know that there exists a solution for all $Y\in\mathcal{M}_{\lambda_2,\lambda_3}$, and $\mathcal{M}_{\lambda_2,\lambda_3}$ has codimension 1 in $\Msymm{3}{R}$. All in all this shows that for all $(\lambda_2,\lambda_3)\in(\R_{>0}\times \R_{>0})\backslash P$ and $Y\in\Msymm{3}{R}$ a solution to Equations \ref{eq-sys-lin-eqs} with $f=f_Y$ exists if and only if $Y\in\mathcal{M}_{\lambda_2,\lambda_3}$. This proves the proposition.
\end{proof}

We conclude by returning to the setting of generators of QMSs. Using the claim and Corollary \ref{corr-lin-sys-eqs} we see that outside of the set $P$, which has measure zero, the corresponding states do not allow a generator $L_Y$, for a symmetric $Y$ with positive entries not satisfying Equation \ref{eq-sol-exists}, to be written as $\delta^*\circ\delta$ for some derivation $\delta$ to a $*$-bimodule. Note that we cannot use Theorem \ref{thm-sys-lin-eqs} to definitively conclude that such a derivation always exists if Equation \ref{eq-sol-exists} is satisfied, because we do not know if the condition that the solution matrix of the system of equations is positive, is satisfied.
\end{example}

\begin{remark}
	A superficial investigation into the $4\times 4$ matrix case indicates that one would find two linear equations, each similar to (\ref{eq-sol-exists}), that describe when a matrix $X$ exists that satisfies the conditions of Theorem \ref{thm-sys-lin-eqs}. Conversely, in the $2\times 2$ case, no equations appear; the matrix $X$ always exists, even though it is not always positive. This suggests that the number of equations increases with the size of the matrices and it gives some intuition as to why we were able to find examples in the $2\times 2$ case where the sought derivation does exist.
\end{remark}

\begin{remark}
	The line of reasoning in Example \ref{examp-detailed} does not work exclusively for the GNS inner product. Any generator given by Theorem \ref{thm-CM} is self-adjoint with respect to $\sip{\cdot,\cdot}_s$ for all $s\in [0,1]$ \cite[Theorem 2.9]{CARLEN20171810}. It is therefore possible to work out the above examples for the inner product $\sip{\cdot,\cdot}_s$ for any $s\in [0,1]$, which gives similar results as above for all values of $s$ that have been tried, apart from one exception. When one considers the $\M{3}{C}$ case for $s=\frac{1}{2}$, then one finds that there always exists a solution $X$ to Equations \ref{eq-sys-lin-eqs}, but that this $X$ is not always positive. Unfortunately, the system of equations does not depend linearly on $s$, so there does not seem to be a clear way to prove that $s=\frac{1}{2}$ is the only special value. However, for any fixed $s$ one can use the above method to investigate when a derivation exists.
\end{remark}
\printbibliography
\end{document}